\documentclass[11pt,reqno]{amsart} 

\usepackage{amssymb,latexsym}
\usepackage{cite} 

\usepackage[height=190mm,width=130mm]{geometry} 

\usepackage{tikz}
\usepackage{pgf,tikz}
\usetikzlibrary{arrows}
\usepackage{calc}
\usepackage{tikz}
\usetikzlibrary{decorations.markings}
\tikzstyle{vertex}=[circle, draw, inner sep=0pt, minimum size=6pt]

\theoremstyle{plain}
\newtheorem{theorem}{Theorem}
\newtheorem{lemma}{Lemma}
\newtheorem{ob}{Observation}

\theoremstyle{definition}
\newtheorem{definition}{Definition}

\theoremstyle{remark}

\numberwithin{equation}{section} 

\begin{document}
\title[On non-cyclic graph of a group]{On non-cyclic graph of a group} 

\author{E. Vatandoost}
\address{IKI University\\ Faculty of Sciences, Department of Mathematics\\
  Qazvin\\ Iran}

\email{vatandoost@sci.ikiu.ac.ir}

\thanks{* We thank Professor Saeed Akbari for his comments and expert of our paper.}

\author{Y. Golkhandypour}

\address{IKI University\\ Faculty of Sciences, Department of Mathematics\\
  Qazvin\\ Iran}


\email{y.golkhandypour@edu.ikiu.ac.ir}

\begin{abstract}
   Here we study some algebraic propertices of non-cyclic graphs. In this paper we show that $\overline{\Gamma}_G$ is isomorphic to $K_3\cup (n-4)K_1$ or $K_4\cup (n-5)K_1$ if and only if $G$ is isomorphic to $D_8$ or $D_{10}$, respectively. We characterize all groups of order $n$ like $G$ in which   $\gamma(\Gamma_G)+\gamma(\overline{\Gamma}_G)\in \{n, n-1, n-2, n-3\}$, where $|Cyc(G)|=1$.
\end{abstract}




\maketitle

 \section{\bf Introduction}
 The study of algebraic structures, using the properties of graphs, becomes an exciting research topic in the last twenty years, leading to many fascinating results and questions. There are many papers on assigning a graph to a ring or group and investigation of algebraic properties of ring or group using the associated graphs.
 \newline
 Before starting let us introduce some necessary notation and definitions.
Let $\Gamma= (V,E)$ be a simple graph. A graph $\Gamma$ is said to be $connected$ if each pair of vertices are joined by a walk. The number of edges of the shortest walk joining $v$ and $u$ is called the $distance$ between $v$ and $u$ and denoted by $d(u, v)$. The maximum value of the distance function in a connected graph $\Gamma$ is called the $diameter$ of $\Gamma$ and denoted by $diam(\Gamma)$.
The degree of a vertex $v\in V(\Gamma)$ is the number of edges which are adjacent to $v$ in $\Gamma$ and denoted by $deg_{\Gamma}(v)$. A universal vertex is a vertex that is adjacent to every other vertices in the graph. The maximum degree of a graph $\Gamma$ denoted by $\Delta(\Gamma)$, and the minimum degree of a graph $\Gamma$ denoted by $\delta(\Gamma)$, are the maximum and minimum degree of its vertices.
\newline
A {\it dominating set} for $\Gamma$ is a subset D of $V(\Gamma)$ such that every vertex not in D is adjacent to at least one vertex of D.
The domination number of $\Gamma$ is the number of vertices in a smallest dominating set for $\Gamma$ and denoted by $\gamma(\Gamma)$. See \cite{111}, \cite{113} and \cite{114} for more details.
\newline
Let $G$ be a non-locally cyclic group and let $Z(G)$ denoted the centralizer of $G$. If $x\in G$, then $Cyc_G(x)$ denotes the cyclicizer of $x$ in $G$ and defined by
\begin{equation*}
 Cyc_G(x)=\{y\in G \mid \langle x,y\rangle~ is~ cyclic\}.
\end{equation*}
The cyclicizer of $G$ denoted by $Cyc(G)$ such that
\begin{equation*}
  Cyc(G)=\{y\in G \mid \langle x,y\rangle ~is ~cyclic~ for ~all~ x\in G\}.
\end{equation*}
It is a simple fact that for any group $G$, $Cyc(G)$ is a locally cyclic subgroup of $G$. Also in general for an element $x$ of a group $G$, $Cyc_G(x)$ is not a subgroup of $G$. Also for each $x\in G$, $C_G(x)$ is used for the centralizer of $x$.
\newline
In this paper, we study some graph properties of the non-cyclic graph $\Gamma_G$ of non-locally cyclic group $G$.
Let $G$ be a non-locally cyclic group of order $n$. The non-cyclic graph of $G$, denoted by $\Gamma_G$, is a graph with vertex set $G\setminus Cyc(G)$ and two vertices $x$ and $y$ are adjacent if and only if $\langle x,y \rangle$ is not cyclic. Also the complement of non-cyclic graph of $G$ is denoted by $\overline{\Gamma}_G$.
Note that the degree of a vertex $x\in V(\Gamma_G)$ in the non-cyclic graph $\Gamma_G$ is equal to $|G\setminus Cyc_G(x)|$.
A. Abdollahi and A. Mohammadi Hassanabadi in \cite{131} show that $\Gamma_G$ is connected and $diam(\Gamma_G)\leq 3$. In addition, they show that $diam(\Gamma_G)=1$ if and only if $G$ is an elementary Abelian $2$-group. Also, in \cite{132} they characterize groups whose non-cyclic graphs have clique numbers at most 4. See \cite{131} and \cite{132} for more details.
\newline
Xuang L. Ma and others in \cite{11111111}, consider The cyclic graph of a finite group $G$. They consider some properties of cyclic graph and characterize certain finite groups whose cyclic graphs have some properties.
\newline
In this paper we show that $\overline{\Gamma}_G$ is isomorphic to $K_3\cup (n-4)K_1$ or $K_4\cup (n-5)K_1$ if and only if $G$ is isomorphic to $D_8$ or $D_{10}$, respectively. Also, we characterize all groups of order $n$ like $G$ in which   $\gamma(\Gamma_G)+\gamma(\overline{\Gamma}_G)\in \{n, n-1, n-2, n-3\}$, where $|Cyc(G)|=1$.  In addition, we prove that $\gamma(\overline{\Gamma}_G)=\frac{n-1}{2}$ if and only if $\overline{\Gamma}_G$ is the union of $\frac{n-1}{2}$ copies of $P_2$.


\section{\bf On non-cyclic graph}
\vskip 0.4 true cm
First we consider some algebraic properties of non-cyclic graph and give some facts that are needed in the next sections. In this paper, we denoted by $A_n$ and $D_{2n}$, alternating group of order $n$ and dihedral group of order $2n$, respectively.
\vskip 0.4 true cm

\begin{theorem} \cite{131} \label{62} Let $G$ be a non-locally cyclic group. Then $\Gamma_G$ is connected and $diam(\Gamma_G)\leq 3$.
\end{theorem}

\begin{theorem} \cite{131} \label{60} Let $G$ be a non-locally cyclic group. Then $diam(\Gamma_G)=1$ if and only if $G$ is an elementary Abelian $2$-group.
\end{theorem}

\begin{lemma}\label{51} Let $G$ be a non-locally cyclic group and $|Cyc(G)|>1$. If $[G:Cyc(G)]=t$, then $t\geq 3$. Furthermore if $|Cyc(G)|\geq 3$, then $\delta(\overline{\Gamma}_G)\geq 2$.
\end{lemma}
\begin{proof}
Let $t=2$. Then there is $g\in G\setminus Cyc(G)$ such that $G=Cyc(G)\cup gCyc(G)$. So $V(\Gamma_G)=gCyc(G)$. Also for each $g_1, g_2\in Cyc(G)$, $\langle gg_1, gg_2\rangle\leq \langle g, g_1, g_2\rangle$ which is a cyclic subgroup of $G$. Hence $\Gamma_G$ is empty which contradicts the fact that $\Gamma_G$ is a connected graph. Therefore $t\geq 3$. If $\ell=|Cyc(G)|\geq 3$, then for each $x\notin Cyc(G)$, induced subgraph on $xCyc(G)$ is isomorphic to $K_{\ell}$ in $\overline{\Gamma}_G$.
\end{proof}

\begin{lemma}\label{92} Let $G$ be a non-locally cyclic group of order $n$ and $|V(\Gamma_G)|=k$. Then $k<n<2k$.
\end{lemma}
\begin{proof}
Let $|V(\Gamma_G)|=k$. Then $n-|Cyc(G)|=k$. By Lemma \ref{51}, $|Cyc(G)|<\frac{n}{2}$. Hence $k<n<2k$.
\end{proof}

\begin{ob}\label{6558} Let $G$ be a non-locally cyclic group and $x\in V(\Gamma_G)$. Then $x^{-1}\in V(\Gamma_G)$ and $N_{\Gamma_G}(x)=N_{\Gamma_G}(x^{-1})$.
\end{ob}

\begin{lemma}\label{58} Let $G$ be a non-locally cyclic group. Then $\delta(\Gamma_G)\geq 2$.
\end{lemma}
\begin{proof}
 Let $x,y\in V(\Gamma_G)$, $x$ be adjacent to $y$ in $\Gamma_G$ and $deg_{\Gamma_G}(x)=1$. We show that $O(y)=2$. On the contrary, let $O(y)\neq 2$. Then $y^{-1}\in V(\Gamma_G)$. Since $\langle x,y \rangle$ is not cyclic, $\langle x,y^{-1} \rangle$ is not cyclic, too. So $x$ and $y^{-1}$ are adjacent in $\Gamma_G$, which contradicts the fact that $deg_{\Gamma_G}(x)=1$. Hence $O(y)=2$.
 \newline
 Since $Cyc(x)\subseteq C_G(x)$ and $|Cyc(x)|>|G|/2$, $C_G(x)=G$ and so $x\in Z(G)$. Also, since $\langle x, y\rangle$ is not a cyclic group, $O(x)$ is even and so $|G|$ is even.
  \newline
 Assume that $g\notin Cyc(G)\cup \langle x\rangle\cup \{y\}$ in that $O(x)=t$ and $O(g)=\ell$. Since $\langle x, g\rangle\cong \mathbb{Z}_t\oplus\mathbb{Z}_{\ell}$ and $\langle x, g\rangle$ is cyclic, $gcd(t, \ell)=1$ and so $\ell$ is odd. Thus $G\setminus Cyc(G)$ has exactly two elements of order two which are $y$ and $x^{t/2}$. It is clear that $O(x^{t/2}y)=2$. Hence $x^{t/2}y=x^{t/2}$ or $x^{t/2}y=y$ and so $y=e$ or $x^{t/2}=e$, respectively, which is a contradiction. Therefore $\delta(\Gamma_G)\geq 2$.
\end{proof}

\begin{lemma}\label{63} Let $G$ be a non-locally cyclic group and $x,y\in V(\Gamma_G)$ such that $O(x)=O(y)=2$. Then $x$ and $y$ are adjacent in $\Gamma_G$.
\end{lemma}

\begin{proof} If $x$ and $y$ are not adjacent in $\Gamma_G$, then $\langle x,y \rangle$ is a cyclic subgroup of order four of $G$, which is wrong.
\end{proof}

\begin{lemma}\label{56} Let $G$ be a non-locally cyclic group and $x\in V(\Gamma_G)$. If $x$ is a universal vertex in $\Gamma_G$, then $O(x)=2$.
\end{lemma}
\begin{proof} On the contrary let $O(x)\neq 2$. Then $x^{-1}\in V(\Gamma_G)$. Since $\langle x,x^{-1} \rangle=\langle x \rangle$ is a cyclic subgroup of $G$, $x$ is not adjacent to $x^{-1}$ in $\Gamma_G$ which contradicts the fact that $x$ is a universal vertex in $\Gamma_G$.
\end{proof}
\begin{lemma}\label{53} Let $G$ be a non-locally cyclic group. Then in $\overline{\Gamma}_G$, the vertices of degree $1$ occur only at the edges.
\end{lemma}
\begin{proof}
Let $x\in V(\Gamma_G)$, $deg_{\overline{\Gamma}_G}(x)=1$ and $x$ is adjacent to $y$ in $\overline{\Gamma}_G$. We claim that $O(x)\in \{3,4\}$. For proof, let $O(x)=2$. Then by Lemma \ref{63}, $O(y)\neq 2$. Since $\langle x,y\rangle$ is cyclic, $\langle x, y^{-1}\rangle$ is cyclic subgroup of $G$, too. Hence $y^{-1}$ is adjacent to $x$ in $\overline{\Gamma}_G$, which contradicts the fact that $deg_{\overline{\Gamma}_G}(x)=1$. Hence $O(x)\neq 2$.
\newline
Obviously, $x^{-1}\notin Cyc(G)$. If $O(x)\geq 5$ and $x^2, x^3\in Cyc(G)$, then $x\in Cyc(G)$, which is false. Thus $x^2$ or $x^3$ are not in $Cyc(G)$ and so $deg_{\overline{\Gamma}_G}(x)\geq 2$, which is a contradiction. Hence $O(x)=3$ or $4$. If $O(x)=4$, then $x^2\in Cyc(G)$ and so $|Cyc(G)|>1$. In any way, $y=x^{-1}$. By observation \ref{6558}, $deg_{\overline{\Gamma}_G}(y)=1$. This completed the proof.
\end{proof}

\begin{lemma}\label{67} Let $G$ be a non-locally cyclic Abelian group. Then $\overline{\Gamma}_G$ does not have both an isolated vertex and a vertex of degree 1.
\end{lemma}
\begin{proof} Let $x$ and $y$ be an isolated vertex and a vertex of degree one in $\overline{\Gamma}_G$, respectively. By Lemmas \ref{56} and \ref{53}, $O(x)=2$ and $O(y)=3$ or $4$. If $O(y)=3$, then since $G$ is an Abelian group, $\langle x,y \rangle\cong \Bbb{Z}_6$. Hence $x$ is adjacent to $y$ in $\overline{\Gamma}_G$, which contradicts the fact that $x$ is an isolated vertex in $\overline{\Gamma}_G$. If $O(y)=4$, then by Lemma \ref{53}, $y^2\in Cyc(G)$. Thus $\langle y^2, x\rangle$ is cyclic. This is contradiction.
\end{proof}
\begin{theorem}\label{59} Let $G$ be a non-locally cyclic group of order $n$ and $|Cyc(G)|=1$. Then $\overline{\Gamma}_G\cong K_3\cup (n-4)K_1$ if and only if $n=8$ and $G\cong D_8$.
\end{theorem}
\begin{proof} If $n=8$ and $G\cong D_8$, then $\overline{\Gamma}_G\cong K_3\cup 4K_1$. For proof of converse, let $\overline{\Gamma}_G\cong K_3\cup (n-4)K_1$. Also let $x_1, x_2, \ldots, x_{n-4}$ be isolated vertices in $\overline{\Gamma}_G$ and $a,b,c\in V(\overline{\Gamma}_G)$ such that induced subgraph on $a, b$ and $c$ be $K_3$. By Lemma \ref{56}, $O(x_i)=2$. It is easy to see that for each $g\in G$, $O(g)\leq 4$ and $ab=ba$, $ac=ca$ and $bc=cb$. We have the following two cases.
\newline
{\bf Case i.} Let $3\mid n$ and $O(a)=3$. Without loss of generality, suppose $b=a^{-1}$ and $O(c)=2$. So $O(ac)=6$ which contradicts the fact that for each $g\in G$, $O(g)\leq 4$.
\newline
{\bf Case ii.} Let $3\nmid n$. By Theorem \ref{60}, there exist $v\in \{a, b, c\}$ such that $O(v)=4$. Without loss of generality, suppose $O(a)=4$, $b=a^{-1}$ and $c=a^{2}$. Since $O(x_iax_i)=O(a)$, $x_iax_i=a$ or $x_iax_i=a^{-1}$. If $x_iax_i=a$, then $x_ia=ax_i$ and $(x_ia)^2=e$. So $e=x_i^2a^2=a^2$, which is a contradiction. Thus $x_iax_i=a^{-1}$ and so $\langle a, x_i \rangle\cong D_8$.
\newline
 We claim that $G\cong \langle a,x_1 \rangle=\{e, a, a^2, a^3, x_1, x_1a, x_1a^2, x_1a^3\}$. For proof of our claim, let $x_2\in G\setminus \langle a,x_1\rangle$.
Since $x_2a\notin \{e, a, a^2, a^3\}$, $O(x_2a)=2$. Thus $x_2ax_2=a^{-1}$. On the other hand, $x_1ax_1=a^{-1}$. So $x_1ax_1=x_2ax_2$. Consequently, $(x_1x_2)a=a(x_1x_2)$ and so $x_1x_2\in C_G(a)$. Since $C_G(a)=\langle a \rangle$, $x_2\in \{x_1, x_1a, x_1a^2, x_1a^3\}$. Hence $x_2\in \langle a,x_1 \rangle$, which is false. Therefore $G\cong \langle a,x_1 \rangle$, as claimed.
\end{proof}
\begin{theorem}\label{66} Let $G$ be a non-locally cyclic group of order $n$ and $|Cyc(G)|=1$. Then $\overline{\Gamma}_G\cong K_4\cup (n-5)K_1$ if and only if $n=10$ and $G\cong D_{10}$.
\end{theorem}
\begin{proof} If $n=10$ and $G\cong D_{10}$, then $\overline{\Gamma}_G\cong K_4\cup 5K_1$. For proof of converse, let $\overline{\Gamma}_G\cong K_4\cup (n-5)K_1$ and $x_1, x_2, \ldots, x_{n-5}$ be isolated vertices in $\overline{\Gamma}_G$ and $a,b,c,d\in V(\overline{\Gamma}_G)$ such that induced subgraph on $a, b,c$ and $d$ be $K_4$. By Lemma \ref{56}, $O(x_i)=2$. It is easy to see that for each $g\in G$, $O(g)\leq 5$.
\newline
By Theorem \ref{60}, there exist $t\in \{a,b,c,d\}$ such that $O(t)\neq 2$. Furthermore, by Lemma \ref{63}, for every $x\in \{a, b, c, d\}$, $O(x)\neq 4$. We consider the following cases.
\newline
{\bf Case i.} Let $u,v\in \{a,b,c,d\}$ such that $O(u)=2$ and $O(v)=3$. Then since $uv=vu$, $O(uv)=6$ which contradicts the fact that for each $g\in G$, $O(g)\leq 5$.
\newline
{\bf Case ii.} Let for each $u\in \{a, b, c, d\}$, $O(u)=3$. Without loss of generality, let $c=a^{-1}$ and $d=b^{-1}$. Then $\langle a,b \rangle\cong \Bbb{Z}_3\oplus \Bbb{Z}_3$. Since $a$ is adjacent to $b$ in $\overline{\Gamma}_G$, $\langle a,b \rangle$ is a cyclic subgroup of $G$, which is false.
\newline
{\bf Case iii.} If $O(a)=5$, then $b=a^2$, $c=a^3$ and $d=a^4$. Since $O(x_iax_i)=O(a)$, $x_iax_i\in \{a, a^2, a^3, a^4\}$. If $x_iax_i\in \{a, a^2, a^3\}$, then $(ax_i)^2\in \{a^2, a^3, a^4\}$. On the other hand, $O(ax_i)=2$ and so $e\in \{a^2, a^3, a^4\}$ which is false. Hence $ax_i=x_ia^4$ and so $\langle a, x_i \rangle\cong D_{10}$.
\newline
By the similar argument as in proof of Theorem \ref{59}, it can be proven that $G\cong \langle a,x_1 \rangle$. Therefore $G$ is isomorphic to $D_{10}$.
\end{proof}

\begin{theorem} \label{52}
Let $G$ be a non-locally cyclic group. Then $\Gamma_G\cong C_n$ if and only if $n=3$ and $G\cong \Bbb{Z}_2\oplus \Bbb{Z}_2$.
\end{theorem}
\begin{proof}
If $G\cong \Bbb{Z}_2\oplus \Bbb{Z}_2$, then $\Gamma_G$ is isomorphic to $C_3$. In the following we consider the proof of converse.
\newline
It is clear that for every $x\in G\setminus \{e\}$, $|C_G(x)|\geq |G|-2$. Then $C_G(x)=G$ and so $G$ is an abelian group. On the other hand, by Theorem \ref{62}, $n\leq 7$. If $n=4$ and $\Gamma_G$ is isomorphic to $C_4$, then by Lemma \ref{92}, $4\leq |G|<8$. Since $G$ is an abelian group, $G$ is cyclic, which is not true.
\newline
Suppose that $5\leq n\leq 7$. By Lemma \ref{63}, suppose that $x\in V(\Gamma_G)$ such that $O(x)\neq 2$. Then $x^{-1}\notin N_{\Gamma_G}(x)$. By Observation \ref{6558}, $N_{\Gamma_G}(x)=N_{\Gamma_G}(x^{-1})$, which is a contradiction.
\end{proof}

\begin{lemma}\label{1037}
  Let $G$ be a non-locally cyclic group and $x,y\in V(\Gamma_G)$ such that $x$ and $y$ are adjacent in $\overline{\Gamma}_G$. Then $|N_{\overline{\Gamma}_G}(x)\cap N_{\overline{\Gamma}_G}(y)|\geq 1$. Furthermore, $\overline{\Gamma}_G$ does not have $C_i$ as a component, where $i\geq 4$.
\end{lemma}
\begin{proof}
  By Lemma \ref{63}, at least $O(x)\neq 2$ or $O(y)\neq 2$. Without loss of generality, let $O(x)\neq 2$. Since $\langle x, y\rangle$ is cyclic group, $\langle x^{-1}, y\rangle$ is cyclic, too. Hence $x^{-1}\in N_{\overline{\Gamma}_G}(x)\cap N_{\overline{\Gamma}_G}(y)$ and so $|N_{\overline{\Gamma}_G}(x)\cap N_{\overline{\Gamma}_G}(y)|\geq 1$. Now it is clear that $\overline{\Gamma}_G$ does not have $C_i$ as a component.
\end{proof}

\section{On acceptable groups}
\vskip 0.4 true cm
In this section, we consider Algebraic structures of non-cyclic graph on acceptable groups and show if $G$ is an acceptable group, then $\overline{\Gamma}_G$ is constructed by both copies of $P_2$ and isolated vertices.
\begin{definition}
  Let $G$ be a group. Write $S$ and $T$ for the set of elements of $G$ of order two and three, respectively. Then $G$ is {\it acceptable} when neither $S$ nor $T$ is empty and $G=S^*\cup T$, where $S^*=S\cup \{e\}$.
\end{definition}
\begin{lemma}\cite{1112223}\label{6556}
  If $G$ is acceptable, then either $S^*\leq G$ or $T^*\leq G$ ($"\leq "$ means "is a subgroup of").
\end{lemma}
\begin{theorem} \label{1112} Let $G$ be a non-locally cyclic group. If $G$ is an acceptable group, then $\overline{\Gamma}_G$ is isomorphic to some copies of $P_2$ together with isolated vertices.
\end{theorem}
\begin{proof}
  Let $G$ be an acceptable group. Write $S( or ~T)$ for the set of elements of $G$ of order two(or three). Also let $x, y\in S$ and $z, t\in T$, where $z\neq t^{-1}$. If $\langle x, z\rangle$ is a cyclic subgroup of $G$, then $O(xz)=6$ which contradicts the fact that $G$ is an acceptable group. Also it is clear that $\langle x,y\rangle$ and $\langle z, t\rangle$ are not cyclic. Hence for each $u\in S$ and $v\in T$, $Cyc_G(u)=\{e,u\}$ and $Cyc_G(v)=\{e, v, v^{-1}\}$. Therefore $|Cyc(G)|=1$ and $\overline{\Gamma}_G$ is isomorphic to some copies of $P_2$ together with isolated vertices.
\end{proof}

\section{\bf On domination number of non-cyclic graph}
\vskip 0.4 true cm
In this section, domination number on some non-cyclic graphs is considered.
\begin{theorem} \cite{113} \label{39} Let $\Gamma$ be a graph with no isolated vertex. Then $\gamma(\Gamma)\leq \frac{n}{2}$.
\end{theorem}

\begin{theorem}\cite{125}\label{c4} Let $\Gamma$ be a graph without isolated vertices on $n$ vertices such that $n$ is even. Then $\gamma(\Gamma)= \frac{n}{2}$ if and only if the components of $\Gamma$ are $C_4$ or $H\circ K_1$ where $H$ is a connected graph.
\end{theorem}

\begin{theorem}\cite{124}\label{Gbar} Let $\Gamma$ be a graph on $n$ vertices. Then
\begin{itemize}
\item[i)] $\gamma(\Gamma)+\gamma(\overline{\Gamma}) \leq n+1$.
\item[ii)]$\gamma(\Gamma)\gamma(\overline{\Gamma}) \leq n$.
\end{itemize}
\end{theorem}

\begin{lemma}\label{19}\cite{123} Let $\Gamma$ be a graph on $n$ vertices. Then $\gamma(\Gamma)=1$ if and only if $\Delta(\Gamma)=n-1$.
\end{lemma}

 \begin{theorem}\label{70} Let $G$ be a non-locally cyclic group of order $n$, where $n$ and $|Cyc(G)|=t$ are odd. Then $\gamma(\overline{\Gamma}_G)=\frac{n-t}{2}$ if and only if $t=1$ and $\overline{\Gamma}_G$ is isomorphic to $\frac{n-1}{2}$ copies of $P_2$.
 \end{theorem}
 \begin{proof}
 Let $\gamma(\overline{\Gamma}_G)=\frac{n-t}{2}$. If $t>1$, then since $t$ is odd, $t>2$. By Lemma \ref{51}, $\delta(\overline{\Gamma}_G)\geq 2$. Also by Theorem \ref{c4}, the components of $\overline{\Gamma}_G$ are $C_4$ or $H\circ K_1$ where $H$ is a connected graph. Hence the components of $\overline{\Gamma}_G$ are $C_4$ which contradicts Lemma \ref{1037}. So $t=1$.
 \newline
 Now let $\gamma(\overline{\Gamma}_G)=\frac{n-1}{2}$. Since $n$ is odd, by Lemma \ref{56}, $\overline{\Gamma}_G$ does not have isolated vertex. By Theorem \ref{c4}, the components of $\overline{\Gamma}_G$ are isomorphic to $C_4$ or $H\circ K_1$. By Lemma \ref{1037}, $\overline{\Gamma}_G$ does not have $C_4$ as a component and so the components of $\overline{\Gamma}_G$ are isomorphic to $H\circ K_1$. By Lemma \ref{53}, $\overline{\Gamma}_G$ is isomorphic to $\frac{n-1}{2}$ copies of $P_2$. The proof of converse is easy.
 \end{proof}
\begin{theorem}\label{55}
Let $G$ be a non-locally cyclic group of order $n$, $|Cyc(G)|=1$. Then $\gamma(\Gamma_G)< \frac{n-1}{2}$.
\end{theorem}
\begin{proof}
Since $\Gamma_G$ is a connected graph, $\gamma(\Gamma_G)\leq \frac{n-1}{2}$. If $\gamma(\Gamma_G)=\frac{n-1}{2}$, then $n$ is odd. By Theorem \ref{c4}, the components of $\Gamma_G$ are isomorphic to $C_4$ or $H\circ K_1$ where $H$ is a connected graph. By Lemma \ref{58}, components of $\Gamma_G$ are of type $C_4$. Since $\Gamma_G$ is a connected graph, $\Gamma_G\cong C_4$. By Theorem \ref{52}, $\Gamma_G$ does not have $C_4$ as a component, which is a contradiction. Therefore $\gamma(\Gamma_G)<\frac{n-1}{2}$.
\end{proof}
\begin{theorem}\label{71} Let $G$ be a non-locally cyclic group of order $n$ and $|Cyc(G)|=1$. Then
\begin{itemize}
\item[i)]$\gamma(\Gamma_G)+\gamma(\overline{\Gamma}_G)=n$ if and only if $G$ is an elementary Abelian $2$-group.
\item[ii)]$\gamma(\Gamma_G)+\gamma(\overline{\Gamma}_G)=n-1$ if and only if $G\cong S_3$.
\item[iii)]$\gamma(\Gamma_G)+\gamma(\overline{\Gamma}_G)=n-2$ if and only if $G\cong D_8$.
\item[iv)]$\gamma(\Gamma_G)+\gamma(\overline{\Gamma}_G)=n-3$ if and only if either $G\cong D_{10}$ or $G\cong \Bbb{Z}_3\oplus \Bbb{Z}_3$.
\end{itemize}
\end{theorem}
\begin{proof}
{\bf i)} Let $\gamma(\Gamma_G)+\gamma(\overline{\Gamma}_G)=n$. By Theorem \ref{55}, $\gamma(\Gamma_G)<\frac{n-1}{2}$. So $\gamma(\overline{\Gamma}_G)>\frac{n-1}{2}$ and so $\overline{\Gamma}_G$ has at least one isolated vertex. Thus $\gamma(\Gamma_G)=1$ and $\gamma(\overline{\Gamma}_G)=n-1$. Hence $\Gamma_G$ is isomorphic to a complete graph. By Theorem \ref{60}, $G$ is an elementary Abelian $2$-group. Conversely, let $G$ be an elementary Abelian $2$-group. By Theorem \ref{60}, $\Gamma_G$ is a complete graph. So $\gamma(\Gamma_G)=1$ and $\gamma(\overline{\Gamma}_G)=n-1$. Thus $\gamma(\Gamma_G)+\gamma(\overline{\Gamma}_G)=n$.
\newline
{\bf ii)} Let $\gamma(\Gamma_G)+\gamma(\overline{\Gamma}_G)=n-1$. By Theorem \ref{55}, $\gamma(\Gamma_G)<\frac{n-1}{2}$. So $\gamma(\overline{\Gamma}_G)>\frac{n-1}{2}$. Hence $\overline{\Gamma}_G$ has at least one isolated vertex. Thus $\gamma(\Gamma_G)=1$ and so $\gamma(\overline{\Gamma}_G)=n-2$. Hence $\overline{\Gamma}_G$ is a union of one copy of $P_2$ and isolated vertices denoted by $\{x_1, x_2, \ldots, x_{n-3}\}$. By Lemma \ref{53}, $G$ has exactly two elements of order $3$ that are denoted by $a$ and $b$ such that $a=b^{-1}$. Obviously, $O(x_iax_i)=O(a)$ and so $x_iax_i=a$ or $a^{-1}$. If $x_iax_i=a$, then $x_ia=ax_i$ and $(x_ia)^2=e$. On the other hand, $(x_ia)^2=x_i^2a^2=a^2$. So $a^2=e$, which is a contradiction. Hence $x_iax_i=a^{-1}$ and so $\langle a,x_i \rangle\cong S_3$. By the similar argument as in proof of Theorem \ref{59}, it can be proven that $G=\langle a,x_1 \rangle\cong S_3$. Conversely, $\gamma(\Gamma_{S_3})=1$ and $\gamma(\overline{\Gamma}_{S_3})=4$. So the result is achieved.
\newline
{\bf iii)} Let $\gamma(\Gamma_G)+\gamma(\overline{\Gamma}_G)=n-2$. By Theorem \ref{55}, $\gamma(\Gamma_G)<\frac{n-1}{2}$. The following two cases will be considered.
\newline
{\bf Case i.} If $\gamma(\Gamma_G)=\frac{n-3}{2}$, then $\gamma(\overline{\Gamma}_G)=\frac{n-1}{2}$. By Theorem \ref{70}, $\overline{\Gamma}_G$ is $\frac{n-1}{2}$ copies of $P_2$ and so $\Gamma_G$ is a complete multipartite graph. Hence $\gamma(\Gamma_G)=2$. It follows that $n=7$, which is impossible.
\newline
{\bf Case ii.} If $\gamma(\Gamma_G)<\frac{n-3}{2}$, then $\gamma(\overline{\Gamma}_G)>\frac{n-1}{2}$ and so $\overline{\Gamma}_G$ has at least one isolated vertex. Thus $\gamma(\Gamma_G)=1$ and $\gamma(\overline{\Gamma}_G)=n-3$. Let $V(\overline{\Gamma}_G)=A\cup B$ in that $A=\{a,b,c,d\}$ and $B=\{x_1\ldots x_{n-5}\}$. We have induced subgraph on $B$ is isomorphic to $\overline{K}_{n-5}$. By Lemmas \ref{53} and \ref{1037} and $\gamma(\overline{\Gamma}_G)=n-3$, induced subgraph on $A$ is isomorphic with two copies of $P_2$ or a union of $K_3$ and an isolated vertex.
\newline
Suppose that induced subgraph on $A$ is isomorphic with two copies of $P_2$. By Lemma \ref{56}, $B$ is a set of elements of order two. Also, by proof of Lemma \ref{53}, $A$ is a set of elements of order $3$.

 Hence $G$ is an acceptable group and so by Lemma \ref{6556}, $A\cup \{e\}\leq G$ or $B\cup \{e\}\leq G$. If $A\cup \{e\}\leq G$, then since $|A\cup \{e\}|=5$, it follows that $3\mid 5$ which is not true. Thus $B\cup \{e\}\leq G$. Since $|B\cup \{e\}|\leq |G|/2$, we have $n\leq 8$. But there is no group of order less than or equal to $8$ such that has four elements of order $3$, which is a contradiction.
\newline
Hence $\overline{\Gamma}_G$ is isomorphic to a union of isolated vertices denoted by $x_1,\ldots, x_{n-4}$
 and three  vertices $\{a,b,c\}$ such that induced subgraph on $\{a, b, c\}$ is isomorphic to $K_3$. By Theorem \ref{59}, $G\cong D_8$. Conversely, $\gamma(\Gamma_{D_8})=1$ and $\gamma(\overline{\Gamma}_{D_8})=5$.
\newline
{\bf iv)} Let $\gamma(\Gamma_G)+\gamma(\overline{\Gamma}_G)=n-3$. By Theorem \ref{55}, $\gamma(\Gamma_G)<\frac{n-1}{2}$. The following three cases will be considered.
\newline
{\bf Case i.} If $\gamma(\Gamma_G)=\frac{n-3}{2}$, then $\gamma(\overline{\Gamma}_G)=\frac{n-3}{2}$. By Theorem \ref{Gbar}, $\gamma(\Gamma_G)\gamma(\overline{\Gamma}_G)\leq n-1$. Hence $(\frac{n-3}{2})^2\leq n-1$ and so $n^2-10n+13\leq 0$ implies that $2\leq n\leq 8$. Since $\gamma(\Gamma_G)=\frac{n-3}{2}$, $n\in \{5,7\}$, which contradicts the fact that $G$ is not a cyclic group.
\newline
{\bf Case ii.} If $\gamma(\Gamma_G)=\frac{n-5}{2}$, then $\gamma(\overline{\Gamma}_G)=\frac{n-1}{2}$. By Theorem \ref{70}, $\overline{\Gamma}_G$ is $\frac{n-1}{2}$ copies of $P_2$ and so $\Gamma_G$ is complete multipartite graph. Hence $\gamma(\Gamma_G)=2$. It follows that $n=9$ and so $G\cong \Bbb{Z}_3\oplus \Bbb{Z}_3$.
\newline
{\bf Case iii.} Let $\gamma(\Gamma_G)<\frac{n-5}{2}$. Then $\gamma(\overline{\Gamma}_G)>\frac{n-1}{2}$ and so $\overline{\Gamma}_G$ has at least one isolated vertex. Thus $\gamma(\Gamma_G)=1$ and $\gamma(\overline{\Gamma}_G)=n-4$.
 \newline
Assume that $u\in V(\overline{\Gamma}_G)$. If $deg(u)>3$, then $\gamma(\overline{\Gamma}_G)< n-4$, which is a contradiction. Hence for each $u\in V(\overline{\Gamma}_G)$, $deg_{\overline{\Gamma}_G}(u)\leq 3$ and so $O(u)\leq 5$. We consider the following subcases.
\newline
{\bf A.} For each $u\in V(\overline{\Gamma}_G)$, $deg_{\overline{\Gamma}_G}(u)\leq 1$.
\newline
In this case, $\overline{\Gamma}_G$ is isomorphic to a union of three copies of $P_2$ and $n-7$ isolated vertices. By Lemmas \ref{56} and \ref{53}, $G$ is a group which is constructed from set of elements of order three and set of elements of order two that are defined by $T$ and $S$, respectively. Hence $G$ is an acceptable group and so by Lemma \ref{6556}, $T\cup \{e\}\leq G$ or $S\cup \{e\}\leq G$. Since $3\nmid |T\cup \{e\}|=7$, $T\cup \{e\}\nleq G$. Hence $S\cup \{e\}\leq G$. Since $|S\cup \{e\}|\leq |G|/2$, we have $n\leq 12$.
\newline
Since there is no group with $n\leq 11$ which have exactly six elements of order $3$, $n=12$. Then $G\cong A_4, D_{12}$ or $L=\langle a, b ~;~ a^6=1, a^3=b^2, b^{-1}ab=a^{-1}$. Since $G$ has exactly six elements of order $3$, $G\notin \{A_4, D_{12}\}$. On the other hand, in $L$, $O(b)=4$, which contradicts the fact that $G$ is an acceptable group.
\newline
{\bf B.} For each $u\in V(\overline{\Gamma}_G)$, $deg(u)\leq 2$.
\newline
Assume that $u,v\in V(\overline{\Gamma}_G)$ such that $deg_{\overline{\Gamma}_G}(u)=deg_{\overline{\Gamma}_G}(v)=2$. If $u$ and $v$ are not adjacent in $\overline{\Gamma}_G$, then $\gamma(\overline{\Gamma}_G)<n-4$, which is wrong. Also, if $u$ and $v$ are adjacent in $\overline{\Gamma}_G$, then by Lemmas \ref{53} and \ref{1037}, $u,v\in V(K_3)$. So $\overline{\Gamma}_G$ is isomorphic to a union of some copies of $K_3$, $P_2$ and isolated vertices. If $\overline{\Gamma}_G$ have more than one copy of $K_3$ as a component, then $\gamma(\overline{\Gamma}_G)\leq n-5$, which is a contradiction.
\newline
Since $\gamma(\overline{\Gamma}_G)=n-4$, $\overline{\Gamma}_G$ is isomorphic to a union of one copy of $C_3$, one copy of $P_2$ and isolated vertices. Suppose that $V(K_3)=\{u_1, u_2, u_3\}$, $V(P_2)=\{v_1, v_2\}$ and $x_i$ is isolated vertex, for $1\leq i\leq n-6$.
By Lemma \ref{63}, we can assume that $O(u_1)\neq 2$. If $O(u_1)=3$ and $u_2=u_1^{-1}$, then $O(u_3)=2$. Since $\langle u_1, u_3\rangle$ is cyclic, there exist an element of order six in $G$ which is a contradiction. Hence $O(u_1)=4$ and so $u_2=u_1^2$ and $u_3=u_1^3$. On the other hand, by Lemma \ref{58}, $O(v_1)=O(v_2)=3$ and by Lemma \ref{56}, $O(x_i)=2$. Therefore $n=2^{\ell}\times 3$.
 \newline
 By Sylow theorem, there exist $H,K\leq G$ such that $|H|=2^{\ell}$ and $|K|=3$. Also $K=\{e, v_1, v_1^{-1}\}$ and $Z(H)\neq 1$. Obviously, for each $h\in G\setminus \{u_1, u_3\}$, $O(hv_1h)=O(v_1)=3$. Thus $hv_1h=v_1$ or $v_1^{-1}$. If $hv_1h=v_1$, then $v_1h=hv_1$ and so $O(v_1h)=6$, which is a contradiction. So
 \begin{equation}\label{178}
  hv_1h=v_1^{-1}~ or ~v_1h=hv_1^{-1}~ or ~hv_1=v_1^{-1}h.
 \end{equation}
Assume that $z\in Z(H)$ and $O(z)=2$. Also let $h\in G\setminus \{u_1, u_3, z\}$. By equalities \ref{178}, we have
\begin{equation*}
  (zh)v_1=(hz)v_1=h(zv_1)=h(v_1^{-1}z)=(hv_1^{-1})z=(v_1h)z=v_1(zh).
\end{equation*}
Hence $zh\in C(v_1)$. On the other hand, $O(zh)=2$ and $C(v_1)=\{e, v_1, v_1^{-1}\}$. So $zh\in \{v_1, v_1^{-1}\}$, which is a contradiction.
\newline
{\bf C.} There exist $u\in V(\overline{\Gamma}_G)$ such that $deg_{\overline{\Gamma}_G}(u)=3$ and $N_{\overline{\Gamma}_G}(u)=\{x,y,z\}$.
{\bf C.1.} If $O(u)=2$, then by Lemma \ref{63}, for each $v\in N(u)$, $O(v)=3$ or $4$. If $O(v)=3$, then $\langle u,v\rangle$ has an element of order six which is false. Thus $O(x)=4$. If $x^{-1}\notin N(u)$, then $\langle x^{-1}, u\rangle$ is cyclic and so $u$ and $x^{-1}$ are adjacent in $\overline{\Gamma}_G$ which is wrong. Hence $x^{-1}\in N(u)$. Without loss of generality, let $y=x^{-1}$. Then $z^{-1}=z$ and so $O(z)=2$, which contradicts Lemma \ref{63}.
\newline
{\bf C.2.} Assume that $O(u)=3$ and without loss of generality, let $x=u^{-1}$. By Observation \ref{6558}, $N_{\overline{\Gamma}_G}(x)=N_{\overline{\Gamma}_G}(u)$. We claim that $z=y^{-1}$. For proof let $z\neq y^{-1}$. Since $\langle u, y^{-1}\rangle$ is cyclic, $u$ is adjacent to $y^{-1}$ in $\overline{\Gamma}_G$, which is false. Hence $z=y^{-1}$, as claimed and so $y$ is adjacent to $z$ in $\overline{\Gamma}_G$. Therefore induced subgraph on $N[u]$ is isomorphic to $K_4$. So $\overline{\Gamma}_G\cong K_4\cup (n-5)K_1$. By Theorem \ref{66}, $G\cong D_{10}$.
\newline
{\bf C.3.} Assume that $O(u)=4$ and without loss of generality, let $x=u^2$ and $y=u^3$. We have
\begin{equation*}
  \langle u^2,z\rangle\leq \langle u,z\rangle.
\end{equation*}
Since $\langle u,z\rangle$ is cyclic, $\langle u^2,z\rangle$ is cyclic, too. If $O(z)=2$, then $\langle u^2,z\rangle$ is Klein four group, which is wrong. Also if $O(z)=3$, then $\langle u^2,z\rangle$ has an element of order six, which is a contradiction.
Suppose that $O(z)=4$ or $5$. By Observation \ref{6558}, $N_{\overline{\Gamma}_G}(z)=N_{\overline{\Gamma}_G}(z^{-1})$. Since $u\in N_{\overline{\Gamma}_G}(z)$, $u\in N_{\overline{\Gamma}_G}(z^{-1})$. So $z^{-1}\in \{u^2, u^3, z\}$. Since $O(z)=4$, $z^{-1}\notin \{u^2, z\}$. Hence $z^{-1}=u^3$ and so $z=u$, which is a contradiction.
\newline
{\bf C.4.} Suppose that $O(u)=5$. Then induced subgraph on $N_{\overline{\Gamma}_G}[u]$ is isomorphic to $K_4$. Hence $\overline{\Gamma}_G\cong K_4\cup (n-5)K_1$. By Theorem \ref{66}, $G\cong D_{10}$.
\newline
The proof of converse is straightforward.
\end{proof}


\section*{Acknowledgement} 

We thank Professor Saeed Akbari\footnote{Department of Mathematical Sciences, Sharif University of Technology, P.O.Box: 11365-9415, Tehran, Iran. Email: $s_{-}akbari@sharif.edu$} for his comments and expert of our paper.

\bibliography{mmnsample}
\bibliographystyle{mmn}

\end{document}